\documentclass[12pt]{amsart}
\usepackage[utf8]{inputenc}
\usepackage[english]{babel}
\usepackage{amsmath,amssymb,amsfonts,amsthm}
\usepackage{amscd} 
\usepackage{mathtools} 
\usepackage{longtable}
\usepackage[noadjust]{cite} 
\usepackage{enumitem}
\usepackage{float}
\usepackage[mathcal]{euscript} 
\usepackage[unicode]{hyperref}
\usepackage{ragged2e}
\usepackage{xfrac} 

\usepackage{graphicx, array}
\graphicspath{{.}{figures/}}
\usepackage[labelfont=small,labelformat=simple]{subcaption}

\usepackage{xcolor}

\theoremstyle{plain}
\newtheorem{theorem}{Theorem}[section]
\newtheorem{lemma}[theorem]{Lemma}
\newtheorem{corollary}[theorem]{Corollary}
\theoremstyle{definition}

\begin{document}

\renewcommand{\datename}{}

\renewcommand{\abstractname}{Abstract}
\renewcommand{\refname}{References}
\renewcommand{\tablename}{Table}
\renewcommand{\figurename}{Figure}
\renewcommand{\proofname}{Proof}

\title[Existence of a smooth circle action near parabolic orbits]{Existence of a smooth Hamiltonian circle action near parabolic orbits}

\author{E. Kudryavtseva$^{\diamond,\star}$}

\author{N. Martynchuk$^{\dagger, \star}$}

\thanks{
Affiliations:\\
$^\diamond$ Faculty of Mechanics and Mathematics, Moscow State University, Moscow 119991, Russia \\
$^\star$ Moscow Center for Fundamental and Applied Mathematics, Moscow, Russia \\
$^\dagger$ Bernoulli Institute for Mathematics, Computer Science and Artificial Intelligence, University of Groningen, P.O. Box 407, 9700 AK Groningen, The Netherlands. \\
\textit{E-mail: eakudr@mech.math.msu.su, n.martynchuk@rug.nl}}

  \begin{abstract}
  We show that every parabolic orbit of a two-degree of freedom integrable system admits a $C^\infty$-smooth Hamiltonian
circle action, which is persistent under small  integrable $C^\infty$ perturbations.
We deduce from this result the structural stability of parabolic orbits and show that they are all smoothly
equivalent (in the non-symplectic sense) to a standard model. Our proof is based on showing that
  every symplectomorphism of a neighbourhood  of a parabolic point preserving the integrals of 
  motion is Hamiltonian whose generating function is smooth and constant on the connected components of the common level sets.

\hspace{-5.7mm} \textit{Keywords}: Liouville integrability; parabolic orbit; circle action; structural stability; normal forms.

\hspace{-5.7mm} \textit{Subject classification}: 37J35, 53D12, 53D20, 70H06
 
\end{abstract}

\maketitle

\section{Introduction}

Parabolic orbits of integrable two-degree of freedom Hamiltonian systems are one of the simplest
examples of degenerate singularities. A typical example of a parabolic orbit is given
by the (singular) fibration:
\begin{equation} \label{eq/fibration}
F = (H, J) \colon S^1 \times \mathbb R^3 \to \mathbb R^2,
\end{equation}
where
$H = x^2 - y^3 + \lambda y$ and $J = \lambda$;
here $x,y,\lambda$ are Euclidean coordinates on $\mathbb R^3$. If $\varphi$ denotes the standard angle coordinate
on $S^1$, then the symplectic structure
can be of the form
$$
\omega_0 = dx\wedge dy + d\lambda\wedge d\varphi
$$
or, more generally, 
\begin{equation} \label{eq/symplectic_structure}
\omega = g(x,y,\lambda)dx\wedge dy + d\lambda\wedge(d\varphi+A(x,y,\lambda)dx + B(x,y,\lambda)dy),
\end{equation}
where $g = g(x,y,\lambda), A= A(x,y,\lambda)$ and $B = B(x,y,\lambda)$ are smooth functions. We note that in this paper,
we consider integrable systems that are (at least) of the $C^\infty$ differentiability class; in particular, the Hamiltonian 
and the first integrals as well as the symplectic form are always assumed to be (at least) smooth.

The singular fibration $F \colon S^1 \times \mathbb R^3 \to \mathbb R^2$ is schematically shown in Fig.~\ref{fig/standard_model}, together with the corresponding \textit{bifurcation diagram}, which is the set of the critical values of $F$, and \textit{bifurcation complex} --- the space of the connected components of $F$.

\begin{figure}[htbp]
\begin{center}
\includegraphics[width=0.94\linewidth]{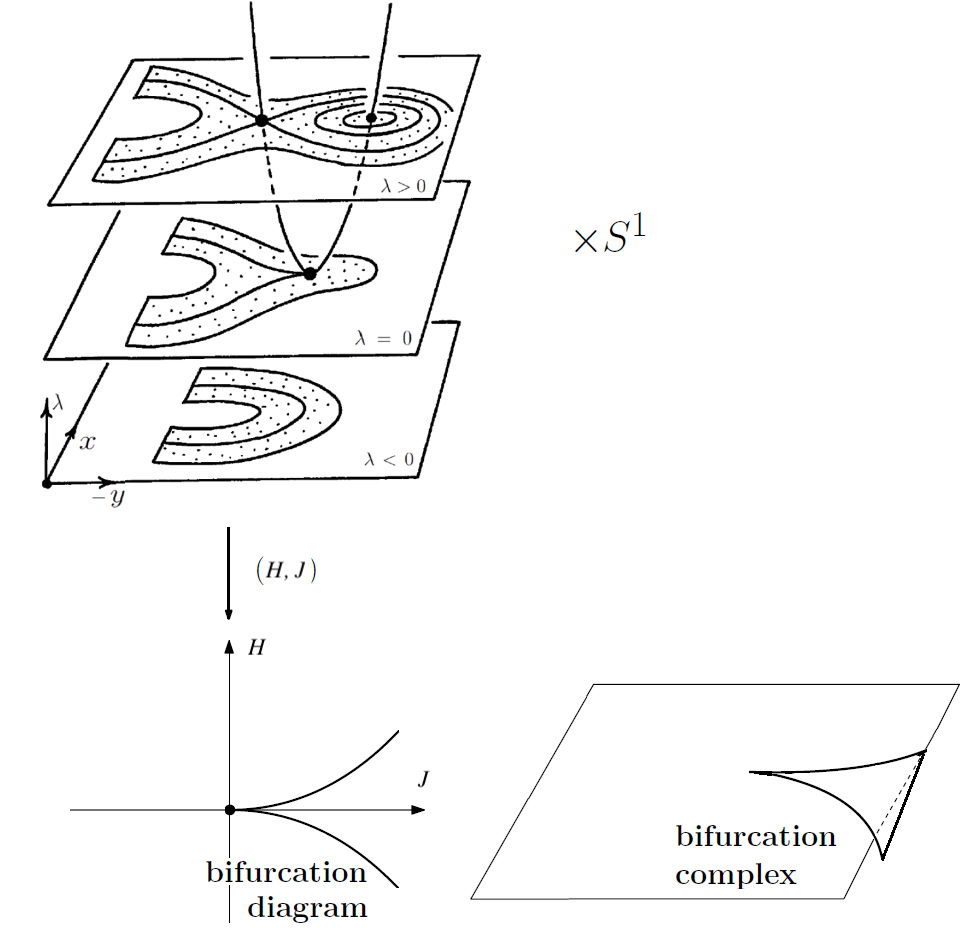}
\caption{The Liouville fibration (top), the bifurcation diagram (bottom left) and the bifurcation complex (bottom right) of the energy-momentum map $(H,J)$}
\label{fig/standard_model}
\end{center}
\end{figure}

As we will show  in this paper, the fibration $F \colon S^1 \times \mathbb R^3 \to \mathbb R^2$ is, in fact, a `standard'
model of a parabolic orbit in the sense that a neighbourhood of such an  orbit can always be put into 
the  form \eqref{eq/fibration}-\eqref{eq/symplectic_structure}. We note that such a
result is well-known when we make the additional assumption that
the integrable system admits a smooth and free Hamiltonian circle action near a parabolic orbit. 
(In the above model \eqref{eq/fibration}-\eqref{eq/symplectic_structure}, the Hamiltonian circle action is given by the 
periodic integral $J$.) It is also known  that parabolic orbits admitting
a Hamiltonian circle action 
are \textit{smoothly structurally stable} in the space of integrable systems 
with such an action; see \cite{Lerman1994, Kalashnikov1998}. The main motivation for the present work is to remove the extra assumption
on the existence of a smooth circle action in the above results.

We note that it is not difficult to show that in a neighborhood of a parabolic orbit, excluding the orbit itself, a smooth circle action always exists. Indeed, it is given by the flow of the periodic first integral
$$
J = \int_{c} \alpha,
$$
where $\alpha$ is a primitive one-form  and $c \subset F^{-1}(f)$ is a cycle  homologous to the parabolic orbit. The smoothness of $J$ (equivalently, of the circle action) follows \cite{Deverdiere1979,Miranda2004} 
from the non-degeneracy of corank 1 singularities on the complement of the parabolic orbit. The main problem is therefore to prove the smoothness  of the periodic integral $J$ near  the parabolic orbit itself. We remark that
in the analytic category,  the corresponding result is known: $J$ and the  circle action are analytic in the case when the integrals and symplectic form are analytic \cite{Zung2000}.
It follows that the analytic equivalence of parabolic orbits and
their analytic structural stability hold
 without the additional assumption on the existence
of a circle action \cite{Zung2000, Kudryavtseva2020, Bolsinov2018}. The same can be said about the topological equivalence
and topological structural stability \cite{Lerman1994, Kalashnikov1998} (one can show this independently, even
without proving the existence of a $C^0$ circle action). What has remained
open until now is whether or not the corresponding results are also true in the smooth $C^\infty$ situation.  

In the present paper, we prove that this is indeed the case. More specifically, we show that 
every parabolic orbit of an integrable two-degree of freedom system admits a smooth free 
Hamiltonian circle action. We deduce from this result that 

i) from the smooth point of view, all parabolic orbits are equivalent, i.e. any two such orbits admit fiberwise diffeomorphic neighbourhoods (which is the direct product of a `standard' 3-dimensional Poincar\'e cross-section and a circle; see Fig.~\ref{fig/standard_model});

ii) parabolic orbits are smoothly  structurally stable 
in the space of all smooth 2-degree of freedom integrable systems (this means that a small integrable $C^\infty$ perturbation of a parabolic singularity is again a parabolic singularity, which is moreover
fiberwise diffeomorphic to the unperturbed one).

The main ingredient in our proof is to show that
any $F$-preserving symplectomorphism of a neighbourhood of a parabolic \textit{point}\footnote{This is a rank-one singular point that locally admits the (non-canonical) coordinates $(x,y,\lambda,\varphi)$ as above.} (and therefore also of a parabolic orbit) is, in fact, Hamiltonian whose generating function is constant on the connected components of the common level sets $\{F= f\}, f \in \mathbb R^2.$  This implies that any such symplectomorphism is smoothly isotopic to the identity in the class of $F$-preserving symplectomorphisms.

We note that a similar result is known for elliptic, non-degenerate corank 1, and focus-focus singularities \cite{Deverdiere1979, Eliasson1990,Miranda2004, VuNgoc2013}. However, it is false in general: the symplectomorphism $(x,y) \mapsto (-y,x)$ of $(\mathbb R^2, dx \wedge dy)$
preserves the function $H = x^4+y^4$, but the corresponding generating function
is not smooth (no even $C^2$ differentiable) at the origin. This means that this symplectomorphism cannot be included into a smooth $H$-preserving Hamiltonian flow. In fact, it cannot be connected to the identity by a smooth (or even $C^3$) $H$-preserving homotopy. This shows that in the context of integrable systems, the problem of the inclusion of a smooth or analytic (symplectic) map into a smooth/analytic flow 
(cf.~\cite{Saulin2016, Pronin1997} and references therein)
does not admit a universal solution, even in the case of polynomial first integrals.

\section{Main results}

In this section, we prove that a neighbourhood of a parabolic orbit of a two degree of freedom system admits a free
Hamiltonian circle action (and, in particular, a periodic integral) in the smooth $C^\infty$ case. Such a result will be used in a subsequent work on the symplectic classification of
parabolic orbits and cuspidal tori in the smooth category \cite{KudryavtsevaMartynchuk2021}; cf.~work \cite{Bolsinov2018} for the analytic case.

Let $(\tilde{H}, G) \colon U \to \mathbb R^2$ be an integrable system with a parabolic orbit $\beta$ (for a formal definition of a parabolic orbit, see \cite[Definition 2.1]{Bolsinov2018}). Assume $dG$ is
non-zero along $\beta$. Then near each point
$P \in \beta$, one can introduce (non-canonical) coordinates $(x,y,\lambda,\varphi) \in D^4$ (note that here $\varphi$ is only a local coordinate) such that
$$
H(\tilde{H},G) = x^2 - y^3 + \lambda y \mbox{ and } G = \lambda
$$
and the symplectic structure has the form \cite{Bolsinov2018}
$$
g(x,y,\lambda)dx\wedge dy + d\lambda\wedge(d\varphi+A(x,y,\lambda)dx + B(x,y,\lambda)dy).
$$
The Hamiltonian flow of $G$ gives rise to the first return map $\mu \colon D^3 \to D^3$, where $D^3$ is a cross-section
given by $\varphi = 0$. The map $\mu$ is smooth. Our goal is to first prove the following.

\begin{theorem} The first return map $\mu$ can be written as
the time-1 map of a smooth Hamiltonian vector field with respect to the symplectic structure 
$g(x,y,\lambda)dx\wedge dy$, where $\lambda$ is regarded as a parameter.
\end{theorem}

\begin{figure}[htbp]
\begin{center}
\includegraphics[width=0.5\linewidth]{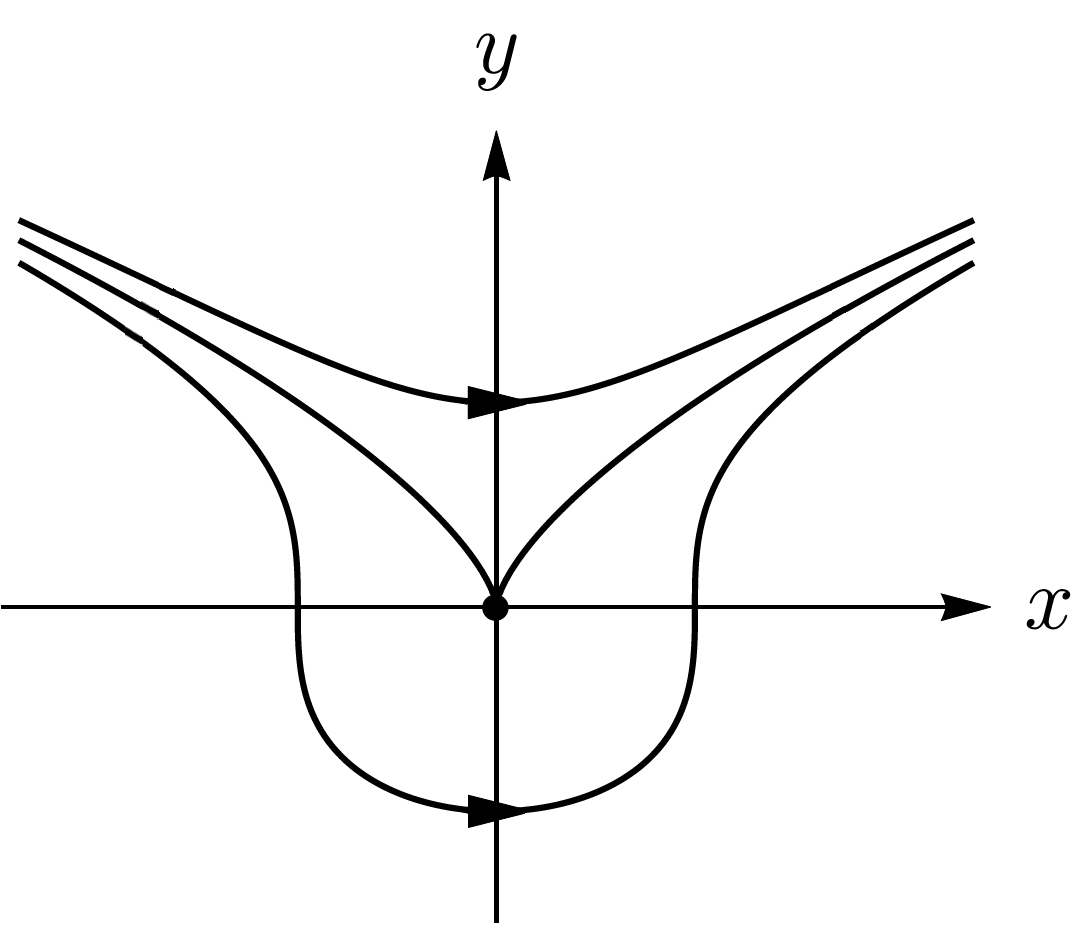}
\caption{The slice of the Liouville fibration for $\lambda = 0$}
\label{fig/lambda0}
\end{center}
\end{figure}

\begin{proof}
Step 1. Consider the  family 
of Lagrangian sections: 
$$L_\lambda = \{(x,y,\lambda) \colon x = 0\}$$
and its image under $\mu$. 
Since $\mu$ is a diffeomorphism preserving the functions $H$ and $G$, the fixed point set of
$\mu$ contains
the parabola $\{x = 0, 3y^2-\lambda = 0\}$; see Fig.~\ref{fig/standard_model}. 

Let $\mu^x$ and
$\mu^y$ denote the $x$- and the $y$-components of $\mu$, respectively. It can be shown
(using that $\mu$  preserves the functions $H$ and $G$, and that the $y$-axis and, hence, its $\mu$-image are `squeezed' between the two branches of the invariant level set $\{\lambda=H=0\}$, see Fig.~\ref{fig/lambda0}) 
that
$\mu^y(0,y,\lambda)$ is monotone 
with respect to $y$ for all small $(y,\lambda)$.
The monotonicity implies that the following formula
$$
u(y,\lambda) = \eta\int_y^{\mu^y(0,y,\lambda)} \dfrac{g(\eta\sqrt{t^3-\lambda t + \lambda y - y^3},t,\lambda)dt}{2\sqrt{t^3-\lambda t + \lambda y - y^3}},
$$
where $\eta = \textup{sign}(\mu^x(0,y,\lambda))$, is well defined for $\lambda \ne 3y^2$; see Fig.~\ref{fig/lambdagreater0}.
We claim that $u = u(y,\lambda)$ extends to a smooth function in a neighbourhood of the origin; this is the content of Lemma below.

\begin{figure}[htbp]
\begin{center}
\includegraphics[width=0.5\linewidth]{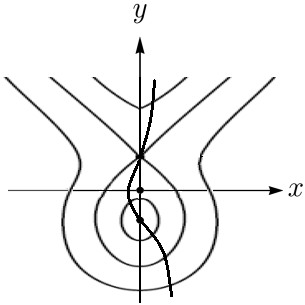}
\caption{The slice of the Liouville fibration for fixed $\lambda > 0$ and the $\mu$-image of the $y$-axis}
\label{fig/lambdagreater0}
\end{center}
\end{figure}

Observe that $u = u(y,\lambda)$ admits a natural extension to a function $\hat{u} = \hat{u}(x,y,\lambda)$ that is
constant on the connected components of $\{H = h, G = \lambda\}$; the function $\hat{u}$ is defined by the condition
$\hat{u}(0,y,\lambda) = u(y,\lambda)$. We claim that $\hat{u} = \hat{u}(x,y,\lambda)$ is also smooth. The required family of Hamiltonian vector fields is then defined by
$$
\hat{u}(x,y,\lambda) X_H,
$$
where $X_H$ denotes the Hamiltonian vector field of the function $H$ with respect to the symplectic structure $g(x,y,\lambda) dx\wedge dy$; recall that here $\lambda$ appears as a parameter.

Step 2. To show $\hat{u} = \hat{u}(x,y,\lambda)$ is smooth, observe that it can be written as $\tilde{u}(H,\lambda)$ and 
$\tilde{u}_\circ(H,\lambda)$ on the closures of each of the two open strata
of the bifurcation complex; see Fig.~\ref{fig/standard_model}.
We will first show\footnote{In fact, to prove $\hat{u}$ is smooth, we will need less information from 
$\tilde{u}_\circ$: it suffices to show $\tilde{u}_\circ$ is smooth for $\lambda > 0$ and that its partial derivatives have a continuous limit at $(H = 0, \lambda = 0)$. We will still need the well-known property of nondegenerate singularities that all partial derivatives of $\tilde{u}$  and $\tilde{u}_{\circ}$ (more precisely, their limits) exist and coincide on the hyperbolic branch $(\lambda > 0, H=2(\lambda/3)^{3/2})$, while all partial derivatives of $\tilde u^\circ$ continuously extend to the elliptic branch $(\lambda>0,\ H=-2(\lambda/3)^{3/2})$ of the bifurcation complex.} in steps 2 and 3
  that the functions $\tilde{u}(H, \lambda)$  and $\tilde{u}_\circ(H, \lambda)$ are smooth on these 
closures (in the sense that each of these functions admit a smooth extension to an open neighbourhood 
of the closure of the corresponding stratum, or equivalently, to $\mathbb R^2$).
Moreover,
we shall show that the corresponding partial derivatives of $\tilde u$ and $\tilde u_\circ$ 
 coincide on the `common boundary' 
$(\lambda\ge0, H=2(\lambda/3)^{3/2})$ of the two strata.

Consider the stratum that is not the swallow-tail domain, and let $\tilde{u}(H, \lambda)$ be the corresponding function defined on it. The smoothness of $\tilde{u}(H, \lambda)$ follows readily from the formula
$$
\tilde{u}(H(\varepsilon,y,\lambda),\lambda)  = \int_y^{\mu^y(\varepsilon,y,\lambda)} \dfrac{g(\sqrt{t^3-\lambda t - y^3+\lambda y},t,\lambda)dt}{2\sqrt{t^3-\lambda t - y^3+\lambda y}};
$$
recall that $\mu^y$ is the $y$-component of $\mu$. Indeed, the right hand side is smooth
as a function of $(y,\lambda)$ since $x = \varepsilon > 0$. Furthermore, at the point 
$(x = \varepsilon, y =  \varepsilon^{2/3}, \lambda = 0)$, we have $H = 0 $, but 
$\partial_yH = \lambda - 3y^2 \ne 0$. So we can take $H$ as a local coordinate instead of $y$.

Now consider the swallow-tail stratum, on which $\tilde{u}_\circ(H, \lambda)$ is defined. Then we have smoothness 
at least in the open half-plane $\lambda > 0$ (in the above sense), since the singularities are $2$D non-degenerate; cf.~\cite{Deverdiere1979} and \cite[Corollary 3.5]{Miranda2004}. Indeed,
near the elliptic family,
this can be shown separately, and near the hyperbolic family, this can be shown using
the Lagrangian section $y = 0$ transversal to the fibers. We note that using the section $y = 0$,
we also have that the partial derivatives of $\tilde{u}$ and $\tilde{u}_\circ$ 
 coincide on the set
$(\lambda > 0, H=2(\lambda/3)^{3/2})$. Let us now prove that the partial derivatives of $\tilde{u}_\circ(H, \lambda)$
extend continuously to the origin. We will then use this to prove $\hat u$ is smooth (and also that the
function $\tilde{u}_\circ(H, \lambda)$ admits a smooth extension, which, as we have noticed earlier,
is not really needed for our purposes).

Step 3. To this end, consider again the case $\lambda > 0$ and observe that
$$
\partial_yu = \left\{ \begin{array}{ll}
\partial_H\tilde u_\circ|_{(-y^3+\lambda y,\lambda)} (-3y^2+\lambda), & -2\sqrt{\lambda/3} < y < \sqrt{\lambda/3}; \\
\partial_H\tilde u|_{(-y^3+\lambda y,\lambda)} (-3y^2+\lambda), & \mbox{otherwise,} \end{array} \right.
$$
where the left hand side is a smooth function for all $(y,\lambda)$ by the Lemma. It follows that 
$$
\partial_yu = 0 \mbox{ for } \lambda = 3 y^2.
$$
Hence, by Hadamard's lemma, $\partial_yu = A(y,\lambda) (\lambda - 3y^2)$ for some smooth function $A = A(y,\lambda)$, which
must then satisfy
$$
A = \left\{ \begin{array}{ll}
\partial_H\tilde u_\circ|_{(-y^3+\lambda y,\lambda)}, & -2\sqrt{\lambda/3} < y < \sqrt{\lambda/3}; \\
\partial_H\tilde u|_{(-y^3+\lambda y,\lambda)}, & \mbox{otherwise.} \end{array} \right.
$$
We thus get that
$\partial_H\tilde{u}_\circ$ extends 
continuously to  $(H = 0, \lambda = 0)$, with the same limit as that of $\partial_H\tilde u$. Similarly one can prove 
the continuity of all partial derivatives. We note that Whitney's extension theorem \cite{Whitney1934} now implies an even stronger
form of differentiability, namely, that $\tilde{u}_{\circ}(H,\lambda)$ admits a smooth extension to an open set, but we do not need this to prove
that $\hat{u}(x,y,\lambda)$ is a smooth function.

Step 4. To show that $\hat{u} = \hat{u}(x,y,\lambda)$ is smooth, it is left to observe that for each
$(x,y,\lambda)$, $\hat{u}(x,y,\lambda) = \tilde{u}(x^2 -y^3+\lambda y, \lambda)$ or $\tilde{u}_\circ(x^2 -y^3+\lambda y, \lambda).$ Indeed, outside the origin
$(0,0,0)$, the smoothness of $\hat{u}$ follows since $\tilde{u}$ and $\tilde{u}_\circ$ are smooth and the restrictions of (the extensions of) the partial derivatives to $(\lambda\ge0, H=2(\lambda/3)^{3/2})$ coincide. Moreover, all of the partial derivatives of $\hat{u}$ will extend continuously to 
$(0,0,0)$ since we have proved in Step 3 that the partial derivatives of $\tilde{u}$ and $\tilde{u}_\circ$ extend continuously to
$(H = 0, \lambda = 0)$. This implies, see for example \cite[Section 3]{Whitney1934},  that $\hat{u} \in C^\infty.$ 
\end{proof}


In Step 1, we used the following lemma.

\begin{lemma} \label{lemma/smoothness}
The function $$
u(y,\lambda) = \eta\int_y^{\mu^y(0,y,\lambda)} \dfrac{g(\eta\sqrt{t^3-\lambda t - y^3+\lambda y},t,\lambda)dt}{2\sqrt{t^3-\lambda t - y^3+\lambda y}},
$$
where $\eta = \textup{sign}(\mu^x(0,y,\lambda))$ and $\lambda \ne 3y^2$, admits a smooth
extension to a neighbourhood of the origin.
\end{lemma}

\begin{proof}

Let $t = y + z^2(\mu^y(0,y,\lambda)-y)$. Denote the difference $\mu^y(0,y,\lambda)-y$ by $\nu$.  Then, for $\nu \ne 0$, 
$$u = \eta\nu\int_0^{1} \dfrac{g(\eta z \sqrt{z^4\nu^3+3z^2\nu^2y+3\nu y^2-\lambda \nu},y+z^2\nu,\lambda)dz}{\sqrt{z^4\nu^3+3z^2\nu^2y+3\nu y^2-\lambda\nu}}.$$
Observe that $\nu (3y^2-\lambda) \ge 0$.
Clearly,
$$z^4\nu^3+3z^2\nu^2y+3\nu y^2-\lambda\nu = \nu (3y^2-\lambda)(1+\dfrac{\nu}{3y^2-\lambda}(z^4\nu+3z^2y))$$
and 
$$
 \dfrac{\eta \nu}{\sqrt{\nu(3y^2-\lambda)}}  = \dfrac{\eta \nu \sqrt{\nu(3y^2-\lambda)}}{\nu (3y^2-\lambda)} = 
 \dfrac{\eta \sqrt{\nu(3y^2-\lambda)}}{3y^2-\lambda}.
$$
Hence for $\lambda \ne 3y^2$ (including the case $\nu = 0, \lambda \ne 3y^2$),
$$u = \dfrac{\eta \sqrt{\nu(3y^2-\lambda)}}{3y^2-\lambda} \int_0^{1} \dfrac{g(\eta z \sqrt{z^4\nu^3+3z^2\nu^2y+3\nu y^2-\lambda \nu},y+z^2\nu,\lambda)dz}{\sqrt{1+\dfrac{\nu}{3y^2-\lambda}(z^4\nu+3z^2y)}}.$$
Now, $\nu = \nu(y,\lambda)$ is a smooth function that is zero on $3y^2=\lambda$. 
By Hadamard's lemma,
$$
\dfrac{\nu}{3y^2-\lambda} \mbox{ and } \big(1+\dfrac{\nu}{3y^2-\lambda}(z^4\nu+3z^2y)\big)^{\pm 1/2},
$$
which are well defined for $\lambda \ne 3y^2$, admit smooth extensions to a small neighbourhood of the origin (when $(y,\lambda)$ are small enough). 

Next, observe that upon substitution of $z = 1$ in the expression
\begin{multline*}
\eta z \sqrt{z^4\nu^3+3z^2\nu^2y+3\nu y^2-\lambda \nu} = \\ \eta z \sqrt{\nu(3y^2-\lambda)} \sqrt{1+\dfrac{\nu}{3y^2-\lambda}(z^4\nu+3z^2y)}
\end{multline*}
we get $\mu^x(0,y,\lambda)$, which is smooth.
It follows that $\eta \sqrt{\nu(3y^2-\lambda)}$ (and hence also the expression itself) is smooth. Moreover,
$\eta \sqrt{\nu(3y^2-\lambda)}$ vanishes when $3y^2-\lambda = 0$ since $\mu^x(0,y,\lambda)$ does. Applying Hadamard's lemma again, we get that
$$
\dfrac{\eta \sqrt{\nu(3y^2-\lambda)}}{3y^2-\lambda}
$$
admits a smooth extension to $\lambda = 3y^2$. We conclude that $u = u(y,\lambda)$ extends to a smooth  function (as a product of functions admitting a smooth extension).
\end{proof}

After we have shown that $\mu$ is the time-1 map of $\hat{u}X_H$, we can consider 
a smooth fiberwise isotopy on $D^3\times[0,\varepsilon] \subset D^4$ 
connecting $\textup{Id}$ with $\mu$ (it is given by the smooth family of vector fields $\alpha(\varphi) \hat{u}X_H$ with $\alpha$ a bump function). This shows the existence of a smooth fibration by circles
of a neighborhood of a parabolic orbit and hence a smooth periodic integral $J$. We have thus proven the following result.

\begin{theorem} \label{theorem/main}
A parabolic orbit of an integrable two-degree of freedom Hamiltonian system $F\colon U\to \mathbb{R}^2$ admits a smooth periodic first integral. More specifically,
there exists a free $F$-preserving $C^\infty$ Hamiltonian circle action in a neighbourhood of such an orbit. \qed
\end{theorem}

\section{Smooth structural stability and normal form}

An important consequence of Theorems~\ref{theorem/main}  is the existence of a smooth (`preliminary') normal form of a parabolic 
singularity. Specifically,
we get the following

\begin{theorem} \label{theorem/normal_form}
Let $F = (\tilde{H},G) \colon U \to \mathbb R^2$ be an integrable two-degree of freedom Hamiltonian system with
a parabolic orbit $\beta$. Then there exist: $(i)$ a small  neighbourhood $V \subset U$ of $\beta$ diffeomorphic to a solid torus $D^3 \times S^1$; 
$(ii)$ smooth functions $H$ and $J$ on $V$ that are constant
on the connected components of $F^{-1}(f)$, and 
$(iii)$ coordinates $(x,y,\lambda,\varphi)$ on $D^3 \times S^1$, with $\varphi$ being an angle coordinate,  such that
$$
H = x^2 - y^3 + \lambda y \mbox{ and } J = \lambda
$$
and the symplectic structure has the form 
$$
\omega = g(x,y,\lambda)dx\wedge dy + d\lambda\wedge(d\varphi+A(x,y,\lambda)dx + B(x,y,\lambda)dy).
$$
\end{theorem}
\begin{proof}
Using the existence of a smooth periodic integral, one can perform the symplectic reduction which reduces the problem to 
 a neighbourhood of a parabolic point, in which case the corresponding results are known; see \cite{Bolsinov2018} and references therein. 
\end{proof}

As a corollary, we get the following stability result.

\begin{corollary} \label{structural/stability}
Let $F  \colon U \to \mathbb R^2$ be an integrable two-degree of freedom Hamiltonian system with
a parabolic orbit $\beta \subset U$. Then every integrable two-degree of freedom system 
$F' \colon U \to \mathbb R^2$ sufficiently close to  $F$ in the $C^\infty$ topology again admits a 
parabolic orbit $\beta' \subset U$ and a smooth periodic integral $J'$. In particular, $F' \colon U \to \mathbb R^2$ is fiberwise $C^\infty$ diffeomorphic
to the unperturbed system in a small neighbourhood of the orbit $\beta'$. \qed
\end{corollary}

\section{Discussion}

In this paper, we have shown that in a neighbourhood of a parabolic point of a two degree of freedom integrable system $F  \colon U \to \mathbb R^2$, 
every $F$-preserving symplectomorphism is Hamiltonian with a smooth generating function that is  constant on the connected components of $\{F= f\}, f \in \mathbb R^2.$ We deduced from this result the existence of a $C^\infty$
Hamiltonian circle action in a neighbourhood of a parabolic singularity as well as a smooth (`preliminary') normal form and structural
stability results; see Theorem~\ref{theorem/normal_form} and Corollary~\ref{structural/stability}.

We conjecture that more is true in fact, and that `uniform' versions of Theorem~\ref{theorem/normal_form} and Corollary~\ref{structural/stability}
hold as well. In particular, this would imply  that the fiberwise diffeomorphism in Corollary~\ref{structural/stability} can be chosen to be close to the identity. These results would follow from a
`uniform' version of the versality theorem and the continuous dependence of the smooth periodic first integral $J$ on the 
system in the $C^\infty$ topology.

\section{Acknowledgements}

The work of the first author was supported by the Russian Science Foundation (grant No. 17-11-01303).

\bibliographystyle{amsplain}
\bibliography{./library.bib}

\end{document}